\documentclass[10pt,twoside]{siamltex}
\usepackage{amsfonts,epsfig}

\newtheorem{conjecture}[theorem]{Conjecture}

\begin{document}



\bibliographystyle{plain}
\title{Extremal graphs for the sum of the two largest signless Laplacian eigenvalues \thanks{Received by the editors on Month x, 200x.
Accepted for publication on Month y, 200y   Handling Editor: .}}

\author{
Carla Silva Oliveira\thanks{Escola Nacional de Ci\^encias Estat\'isticas (ENCE/IBGE), Rio de Janeiro, Brasil (carla.oliveira@ibge.gov.br). Partially supported by grant 305454/2012-9, CNPq, Brazil.}
\and
Leonardo de Lima\thanks{Centro Federal de Educa\c c\~ao Tecnol\'ogica Celso Suckow da Fonseca (CEFET-RJ), Rio de Janeiro, Brasil (leolima.geos@gmail.com). Partially supported by grant 305867/2012-1, CNPq, Brazil}
\and
Paula Rama\thanks{Center for Research and Development in Mathematics and Apllications, Department of Mathematics, University of Aveiro, Aveiro, Portugal (prama@mat.ua.pt), (paula.carvalho@ua.pt). Supported in part by \textit{FEDER} funds through \textit{COMPETE} - Operational Programme Factors of Competitiveness ("Programa Operacional Fatores de Competitividade") and by Portuguese funds through the \textit{Center for Research and Development in Mathema\-tics and Applications} and the Portuguese Foundation for Science and Technology ("FCT-Funda\c c\~ao  para Ci\^encia e a Tecnologia"), within project PEst-C/MAT/UI4106/2011 with COMPETE number FCOMP-01-0124-FEDER-022690.}
\and
Paula Carvalho\footnotemark[4]}
%


\pagestyle{myheadings}
\markboth{C. S. \ Oliveira, L.\ de Lima, P. \ Rama and P.\ Carvalho}{C. S. \ Oliveira, L.\ de Lima, P. \ Rama and P.\ Carvalho}
\maketitle

\begin{abstract}
Let $G$ be a simple graph on $n$ vertices and $e(G)$ edges. Consider $Q(G)=D+A$
as the signless Laplacian of $G$, where $A$ is the adjacency matrix and
$D$ is the diagonal matrix of the vertices degree of $G.$ Let $q_{1}(G)$ and $q_2(G)$ be the first and the second largest eigenvalues of $Q(G),$ respectively, and denote by $S_{n}ˆ{+}$ the star graph plus one edge. In this paper, we prove that inequality $q_1(G) + q_2(G) \leq e(G)+3$ is tighter for the graph $S_{n}^{+}$ among all firefly graphs and also tighter to $S_{n}^{+}$ than to the graphs $K_{k} \vee \overline{K}_{n-k}$ recently presented by Ashraf,  Omidi and Tayfeh-Rezaie. Also, we conjecture that the same inequality is tighter to $S_{n}^{+}$ than any other graph on $n$ vertices.
\end{abstract}


\section{Introduction} \label{intro-sec}

Given a simple graph $G$ with vertex set $V(G)$ and edge set $E(G)$, write $A$ for the adjacency matrix of $G$ and let $D$ be
the diagonal matrix of the row-sums of $A,$ i.e., the degrees of $G.$ The maximum degree of $G$ is denoted by $\Delta = \Delta(G).$ Let $e(G)=|E(G)|$ be the number of edges and let $n = |V(G)|$ be the number of vertices of $G$. The matrix $Q\left(  G\right)  =A+D$ is called the \emph{signless Laplacian} or the $Q$-matrix of $G$. As usual, we shall index the eigenvalues of $Q\left(  G\right)  $ in non-increasing order and denote them as $q_{1}\left(  G\right)  ,q_{2}\left( G\right)  ,\ldots,q_{n}\left(  G\right)$. Denote the star graph on $n$ vertices plus one edge by $S_{n}^{+},$ $\overline{G}$ as the complement graph of $G$ and $K_n$ as the complete graph on $n$ vertices. If $G_1 = (V_1, E_1)$ and $G_2 = (V_2, E_2)$ are graphs on disjoint sets of vertices, their graph sum is $G_1 + G_2 = (V_1 \cup V_2, E_1 \cup E_2).$ The join $G_1 \vee G_2$ of $G_1$ and $G_2$ is the graph obtained from $G_1 + G_2$ by adding new edges from each vertex in $G_1$ to every vertex of $G_2.$ Consider $M(G)$ as a matrix of a graph $G$ of order $n$ and let $k$ be a natural number such that $1 \leq k \leq n$. A general question related to $G$ and $M(G)$ can be raised: \emph{"How large can be the sum of the $k$ largest eigenvalues of $M(G)$ ?"} Usually, solving cases $k=1, n-1$ and $n$ are simple but the general case for any $k$ is not easy to be solved. The natural next case to be studied is $k=2$ and some work has been recently done in order to prove this case. For instance,  Ebrahimi \emph{et al.}, \cite{EMNA08}, for the adjacency matrix; Haemers \emph{et al.}, \cite{HMT10}, for the Laplacian matrix and Ashraf \emph{et al.}, \cite{AOT13}, for the signless Laplacian matrix. In particular, the latter denoted the sum of the two largest signless Laplacian by $S_2(G)$ and proved that
\begin{equation}
S_2(G) \leq e(G)+3, \label{ineq1}
\end{equation}
for any graph $G.$ Besides, they proved that the inequality (\ref{ineq1}) is asymptotically tight for graphs of the type $K_{k}\vee \overline{K}_{t}$, where $k+t=n.$ Given a graph $G$ with $e(G)$ edges, define the function $$f(G)=e(G)+3-S_2(G).$$ Since inequality (\ref{ineq1}) is asymptotically tight for the graphs $K_{k}\vee \overline{K}_{t}$, it means that $f(K_{k}\vee \overline{K}_{t})$ converges to zero when $n$ goes to infinity. In this paper, we proved the following facts: (A) the function $f(S_{n}^{+})$ converges to zero when $n$ goes to infinity and the graph $S_{n}^{+}$ is the only one such that inequality (\ref{ineq1}) is asymptotically tight within the firefly graphs; (B) the function $f(S_{n}^{+})$ converges to zero faster than $f(K_{k}\vee \overline{K}_{t}).$  Besides, based on computational experiments from AutoGraphiX \cite{CH00}, we conjecture that $f(S_{n}^{+})$ converges to zero faster than $f(G)$ when $G$ is any graph on $n$ vertices.


\section{Preliminaries}\label{section2}

In this section, we present some known results about $q_1(G)$, $q_2(G)$ and define some classes of graphs that will be useful to our purposes.

\begin{definition}
A firefly graph $F_{r,s,t}$ is a graph on $2r + s + 2t + 1$ vertices that consists of $r$ triangles, $s$
pendant edges and $t$ pendant paths of length $2$, all of them sharing a common vertex.
\end{definition}

Let $v$ be a vertex of $G$ and let $P_{q+1}$ and $P_{r+1}$ be two paths, say, $v_{q+1}v_{q}\ldots v_{2} v_{1}$ and $u_{r+1}u_{r}\ldots u_{2} u_{1}$.  The graph $G_{q,r}$ is obtained by vertex coalescence of $v_{q+1}$ and $u_{r+1}$ at the same vertex $v$ of $G.$

\begin{definition}[\cite{C10}, Grafting an edge]
Let $G_{q,r}$ be a graph. The graph $G_{q+1,r-1}$ is obtained from $G_{q,r}$ by removing the edge $(u_1,u_2)$ and placing the edge $(v_1,u_1).$
\end{definition}

If $G$ is connected with $e(G) = n+c-1,$ then $G$ is called a $c-$cyclic graph.

\begin{lemma}[\cite{M11}]
Suppose $c \geq 1$ and $G$ is a $c-$cyclic graph on $n$ vertices with $\Delta \leq n-3.$ If $n \geq 2c+5,$ then $q_1(G) \leq n-1.$
\label{teo_liu}
\end{lemma}

\begin{lemma}[\cite{AHL11}]
Let $G$ be a connected graph on $n \geq 7$ vertices. Then
\begin{itemize}
\item[(i)] $3-\frac{2.5}{n} < q_2 < 3$ if and only if $G$ is a firefly with one triangle.
\item[(ii)] $q_2=3$ if and only if $G$ is a firefly and has at least two triangles.
\end{itemize}
\label{teo_hansen}
\end{lemma}

\begin{lemma}[\cite{CRS07}]
Let $G$ be a connected graph on $n \geq 2$ vertices. For $q\geq r \geq 1$, consider the graphs $G_{q,r}$ and $G_{q+1,r-1}.$ Then, $$q_{1}(G_{q,r}) > q_{1}(G_{q+1,r-1}).$$
\label{grafting}
\end{lemma}


\section{Main results}\label{section3}

In this section, we present the proofs of facts (A) and (B) presented in the introduction. In order to prove fact (A), we firstly present Lemma \ref{sn+}. From this point we will use $F_{1,n-3,0}$ to denote the graph $S_{n}^{+}$ since they are isomorphic.

\begin{lemma} \label{sn+} Let $G$ be isomorphic to $F_{1,n-3,0}$ such that $n \ge 7$. Then $$e(G)+3 -\frac{2.5}{n}< S_2(G) < e(G)+3.$$
\end{lemma}

\begin{proof}
Consider a graph $G$ isomorphic to $F_{1,n-3,0}$. The matrix $Q(G)$ can be written as

$ Q(G)=\left[ \begin{array}{c|c|ccc}
   {I} + {J} & {1}  &     & {0} & \\
\hline                     {1^T} & n-1 &     & {1} & \\
\hline                          &             &  &            &  \\
         {0^T}             & {1^T}  &    &   {I}   \\

\end{array} \right],$
where the diagonal blocks are of orders 2, 1 and $n-3$, respectively. For each $i= 1, \ldots, n,$ let $e_i$ denote the ith standard unit basis vector. We find that $e_1-e_2$ and $e_4-e_j$, $5 \leq j \leq n$ are eigenvectors for $Q(G)$ corresponding to the eigenvalue $1$. Consequently, we see that $Q(G)$ has $1$ as an eigenvalue of multiplicity at least $n-3$. Further, since $Q(G)$ has an orthogonal basis of eigenvectors, it follows that there are remaining eigenvectors of $Q(G)$ of the form $\left[\begin{array}{c} \alpha {\bf{1}}\\ \hline \beta  \\ \hline \gamma {\bf{1}} \end{array}\right].$ We then deduce that the eigenvalues of the $3 \times 3$ matrix $M= \left[ \begin{array}{ccc} 3 &1&0\\2&n-1&n-3\\ 0&1&1\end{array}\right]$ comprise the remaining three eigenvalues of $Q(G)$ that are the roots of polynomial $\Psi(x) = x^3 -(n+3)x^2 +3n x-4$. As $\Psi$ is a continuous function in $\mathbb{R}$ and $\Psi(0)= -4 <0$, $\Psi(1)= 2n-6 > 0,$ from \cite{AHL11}, $\Psi(3-\frac{2.5}{n}) > 0,$ $\Psi(3-\frac{1}{n})= -1 - \frac{1}{n^3}+\frac{6}{n^2}-\frac{10}{n} < 0,$ $\Psi(n)=-4<0,$ and  for $ n \ge 7$,  $\Psi(n+ \frac{1}{n}) = -7+\frac{1}{n^3}-\frac{3}{n^2}+\frac{2}{n}+n> 0$, so $3-\frac{2.5}{n} < q_2(G) < 3-\frac{1}{n}$ and $ n < q_1(G) < n+ \frac{1}{n}$ for $ n \ge 7$. Then $e(G)+3 -\frac{2.5}{n} < S_2(G) < e(G)+3$.
\end{proof}

From Lemma \ref{sn+}, one can easily see that function $f(F_{1,n-3,0})$ converges to zero when $n$ goes to infinity. To complete the prove of the statement (A) we need to show that $F_{1,n-3,0}$ is the only graph such that inequality (\ref{ineq1}) is asymptotically tight within the firefly graphs. The prove follows from Lemmas \ref{Hn}, \ref{Fn22} and \ref{Fn21}.


\begin{lemma} \label{Hn}
Let $G$ be isomorphic to $F_{1,n-5,1}$ such that $n\ge 9$. Then $$e(G)+2 - \frac{0.8}{\ln n} < S_2(G) < e(G)+2.$$
\end{lemma}

\begin{proof}
Consider a graph $G$ isomorphic to $F_{1,n-5,1}.$ The matrix $Q(G)$ can be written as

$Q(G)=\left[ \begin{array}{ccccccccc} 2&1&1&0&0&0&\ldots
&&0 \\1&2&1&0&0&0&\ldots&&0\\1&1&n-2&1&0&1&\ldots&&1\\  0&0&1&2 &1 &0& \ldots &&0\\ 0&0& 0& 1&1& 0&\ldots &&0\\
0&0&1&0&0&1&\ldots &&0\\ \vdots& \vdots& \vdots&\vdots&\vdots&& \ddots&&\\ 0&0&1&0&0&&\ldots&&1
\end{array}\right].$
For each $i= 1, \ldots, n,$ let $e_i$ denote the ith standard unit basis vector. We find that $e_6 - e_j$, $7 \leq j \leq n$ and $e_1 - e_2$ are eigenvectors for $Q(G)$ corresponding to eigenvalue $1$. Consequently, we see that $Q(G)$ has $1$ as an eigenvalue of multiplicity at least $n-5$. Further, since $Q(G)$ has an orthogonal basis of eigenvectors, it follows that there are remaining eigenvectors of $Q(G)$ of the form $\left[\begin{array}{c} \alpha {\bf{1}} \\ \hline \beta  \\ \hline \gamma \\ \hline \xi \\ \hline \varpi {\bf{1}} \end{array}\right].$ We then deduce that the eigenvalues of $5 \times 5$ matrix $M= \left[ \begin{array}{ccccc} 3 &1&0&0&0 \\2&n-2&1&0&n-5\\ 0&1&2&1&0 \\ 0&0&1&1&0 \\ 0&1&0&0&1\end{array}\right]$
comprise the remaining five eigenvalues of $Q(G)$ that are the roots the polynomial $\Psi(x)=x^5 - (n + 5) x^4 + (6 n + 4) x^3 - (10 n - 2) x^2 + (3 n + 12) x - 4$. As $\Psi$ is a continuous function in $\mathbb{R}$, there are three roots of $\Psi(x)$ in the intervals $[0,0.3]$, $[0.3,1]$ and $[1,2.7]$. For the other two, as
\begin{eqnarray*}
\Psi\left(3 - \frac{0.8}{\ln n}\right)&=&
    -4+(12+3 n) \left(3-\frac{0.8}{\ln n}\right)-(-2+10 n) \left(3-\frac{0.8}{\ln n}\right)^2\\
    &&+(4+6 n) \left(3-\frac{0.8}{\ln n}\right)^3-(5+n) \left(3-\frac{0.8}{\ln n}\right)^4+\left(3-\frac{0.8}{\ln n}\right)^5\\
    & >&0,
\end{eqnarray*}
\begin{eqnarray*}
\Psi\left(3 -\frac{5}{4n}\right)&=&
    -\frac{1}{1024 n^5}\left(3125 - 25000 n + 70500 n^2 - 72800 n^3 + 12160 n^4 + 256 n^5\right)\\
    & <&0,
\end{eqnarray*}
$$\Psi (n-1)= -24 + 25 n - 5 n^2 <0$$ and
\begin{eqnarray*}
\Psi (n - 1 + \frac{5}{4 n})&=&
    \frac{1}{1024 n^5} \left(3125 - 25000 n + 78000 n^2 - 140000 n^3 + 178400 n^4 \right.\\
    &&\left. - 142976 n^5 + 75520 n^6 - 17920 n^7 + 1280 n^8\right)\\
 &>& 0
\end{eqnarray*}
so, $ 3-\frac{0.8}{\ln n} < q_2(G) < 3-\frac{5}{4n}$ and $n-1 < q_1(G) < n - 1 + \frac{5}{4 n}$. Then, $e(G) + 2 - \frac{0.8}{\ln n} < S_2(G) < e(G) + 2.$
\end{proof}

From Lemma \ref{Hn}, one can easily see that function $f(F_{1,n-5,1})$ converges to 1 when $n$ goes to infinity.


\begin{lemma}\label{Fn22}
Let $G$ be isomorphic to $F_{1,s,t}$ a firefly graph such that $s \geq 1$ and $t \geq 2.$ Then $S_2(G) < e(G)+2.$
\end{lemma}

\begin{proof}
From Lemma \ref{teo_liu}, we have $q_1 \leq s+2t+2$ and from Lemma \ref{teo_hansen}, $q_2 < 3$.
So, $S_2(G) < s+2t+5 = e(G)+2$.
\end{proof}

From Lemma \ref{Fn22}, follows that function $f(F_{1,s,t}) > 1$ when $s \geq 1$ and $t \geq 2.$



\begin{lemma} \label{Fn2}
For  $2r+s+1\ge 6$ and $r \ge 2$, then $$2r+s+1 < q_1(F_{r,s,0}) < 2r + s+\frac{3}{2}.$$
\end{lemma}

\begin{proof}
The signless Laplacian matrix of the graph $F_{r,s,0}$ can be written as
$$ Q(F_{r,s,0})=\left[ \begin{array}{c|c|c}
   2r+s    & \mathbf{1}     & \mathbf{1}                    \\  \hline
   {1^T}  &  \mathbf{I}    & \mathbf{0}       \\ \hline
   {1^T}  & \mathbf{0}  &  \mathbf{B}   \\
\end{array} \right],$$
where the diagonal blocks are of orders $1,\; s$ and $2r,$ respectively and $B$ is a diagonal block matrix and each block has order $2$ of the type $I+J.$ For each $i=1, \ldots, 2r+s+1,$ let $e_i$
denote the $i$-th standard unit basis vector. We find that for each $j=3,\ldots, s+1, e_2-e_j$ is an eigenvector for $Q(F_{r,s,0})$ corresponding to eigenvalue $1$; also, for each $k=1,\ldots, r, e_{s+2k}-e_{s+2k+1}$ is an eigenvector for $Q(F_{r,s,0})$ corresponding to eigenvalue $1$. So, $1$ is an eigenvalue with multiplicity at least $r+s-1.$ Further, since  $Q(F_{r,s,0})$ has an orthogonal basis of eigenvectors, it follows that there are remaining eigenvectors of  $Q(F_{r,s,0})$ of the form $\left[\begin{array}{c} \gamma \\ \hline \alpha {\bf{1}}\\ \hline \beta {\bf{1}} \end{array}\right].$  We then deduce that the eigenvalues of the $3 \times 3$ matrix $$M= \left[ \begin{array}{ccc} 2r+s &s&2r\\1&1&0\\ 1&0&3\end{array}\right]$$ comprise the remaining three eigenvalues of $Q(F_{r,s,0})$. The eigenvalues of $M$ are the roots of the characteristic polynomial of $M$ given by $g(x) = -(x^3+(-s-2r-4)x^2 + (3s+6r+3)x-4r).$ See that $g(2r+s+1) >0$ and $g(2r+s+3/2) <0.$ Since $q_2 \leq n-2 = 2r+s-1$ and $q_1 \geq q_2,$ we get $$2r+s+1 < q_1 < 2r+s+\frac{3}{2}.$$
\end{proof}

\begin{lemma}\label{Fn21}
Let $G=F_{r,s,t}$ such that $r \ge 2,$ $t,s \geq 1$. Then $S_2(G) \leq e(G) + 2.5.$
\end{lemma}
\begin{proof}
For firefly graphs $F_{r,s,t}$ such that $t \geq 1$, we can obtain any $F_{r,s,t}$ from grafting edges of the graph $F_{r,s,0}$. From  Lemma \ref{grafting} and Lemma \ref{Fn2} $q_1(F_{r,s,t}) < q_1(F_{r,s + 2t,0}) < 2r+s+2t +\frac{3}{2}.$ Also, by Lemma \ref{teo_hansen}, $q_2(F_{r,s,t})=3$ we get $ q_1+q_2 < 2r+s+2t+4.5.$ Observe that $e(F_{r,s+2t,0}) = 3r+s+2t$ and then $2r+s+2t+4.5 = e(F_{r,s+2t,0})+2.5+(2-r) \leq e(F_{r,s+2t,0})+2.5$ for $r \geq 2$. So, $S_2(G) \leq e(G) + 2.5.$
\end{proof}

From Lemma \ref{Fn21}, follows that function $f(F_{r,s,t}) \geq 0.5$ when $r \geq 2.$ The next proposition proves the statement (B) of the introduction.

\begin{proposition}
For $n \geq 9$ and $k \geq 2$, the function $f(F_{1,n-3,0})$ converges to zero faster than $f(K_{k}\vee \overline{K}_{n-k}).$
\end{proposition}

\begin{proof}
From Lemma \ref{sn+}, we have $ 0<f(F_{1,n-3,0})<\frac{2.5}{n}$ and from Remark 8 of \cite{AOT13}, $0< f(K_{k}\vee \overline{K}_{n-k})< \frac{1}{\sqrt{n-k}}.$ One can see that $\frac{2.5}{n} < \frac{1}{\sqrt{n-k}}$ for $k \geq 2$ and it completes the proof.
\end{proof}

Therefore, we proved that inequality (\ref{ineq1}) is asymptotically tight for the graph $F_{1,n-3,0}$ within the firefly graphs on $n \geq 9$ vertices. Based on computational experiments we propose the following conjecture.

\begin{conjecture}
For any graph $G$ on $n \geq 9$ vertices, $f(F_{1,n-3,0})$ converges to zero faster than $f(G).$
\end{conjecture}


\end{document}